\documentclass[12pt]{amsart}

\usepackage{hyperref}
\usepackage{fullpage}
\usepackage{amsmath}
\usepackage{amsthm}
\usepackage{mathrsfs}
\usepackage{amssymb}
\usepackage{comment}
\newtheorem{theorem}{Theorem}[section]
\newtheorem{lemma}[theorem]{Lemma}

\theoremstyle{definition}
\newtheorem*{remark}{Remark}

\newcommand{\fld}{\mathbb{F}}

\begin{document}
\title{An explicit approach to the Ahlgren-Ono conjecture}
\author{Geoffrey D. Smith}
\address{Department of Mathematics, Yale University, 10 Hillhouse Avenue New Haven, Connecticut, 06511}
\email{geoffrey.smith@yale.edu}
\author{Lynnelle Ye}
\address{Department of Mathematics, Harvard University, Cambridge, MA 02138}
\email{lye@g.harvard.edu}

\begin{abstract}
Let $p(n)$ be the partition function. Ahlgren and Ono conjectured that every arithmetic progression contains infinitely many integers $N$ for which $p(N)$ is not congruent to $0\pmod{3}$. Radu proved this conjecture in 2010 using work of Deligne and Rapoport. In this note, we give a simpler proof of Ahlgren and Ono's conjecture in the special case where the modulus of the arithmetic progression is a power of $3$ by applying a method of Boylan and Ono and using work of Bella\"iche and Khare generalizing Serre's results on the local nilpotency of the Hecke algebra.
\end{abstract}

\maketitle

\section{Introduction and statement of results}\label{s1}

A \emph{partition} of a nonnegative integer $n$ is a non-increasing sequence of positive integers whose sum is $n$. The partition function $p(n)$ then counts the number of distinct partitions of $n$. The generating function for $p(n)$ was shown by Euler to be
\[
\sum_{n=0}^\infty p(n)q^n=\prod_{n=1}^\infty\frac1{1-q^n}=1+q+2q^2+3q^3+5q^4+7q^5+\dotsb.
\]
Ramanujan famously observed the congruences $p(5k+4)\equiv 0 \pmod 5$, $p(7k+5)\equiv 0 \pmod 7$, and $p(11k+6)\equiv 0\pmod{11}$. More recently, Ono in \cite{ono-annals} and Folsom, Kent, and Ono in \cite{folsom-kent-ono}
have used Serre's theory of modular forms modulo $p$ to prove general results about the $p$-adic behavior of $p(n)$ for all $p\geq 5$. As a consequence, the behavior of $p(n)$ is well understood modulo $p$ for all $p\geq5$.


The behavior of $p(n)$ modulo 2 and 3 is far less well understood. It is widely believed that $p(n)$ is equidistributed modulo $2$ and $3$, but little is known. Subbarao conjectured in~\cite{subbarao} that for any arithmetic progression $B\pmod{A}$, there are infinitely many integers $N\equiv B\pmod{A}$ for which $p(N)$ is even, and also infinitely many such $N$ for which $p(N)$ is odd. In~\cite{ahlgren-ono}, Ahlgren and Ono conjectured that for any arithmetic progression $B\pmod{A}$, there are infinitely many integers $N\equiv B\pmod{A}$ for which $p(N)\not\equiv0\pmod{3}$. 

Ono \cite{ono} established half of Subbarao's conjecture, proving that there are infinitely many $N$ in every arithmetic progression for which $p(n)$ is even. Boylan and Ono \cite{boylan-ono} then used the local nilpotency of the Hecke algebra, as observed by Serre in \cite[p. 115]{serre}, to prove the odd case of Subbarao's conjecture in the case $A=2^s$. Radu proves the full conjecture, along with Ahlgren and Ono's for modulo $3$, in~\cite{radu}, using work \cite{deligne-rapoport} of Deligne and Rapoport that applies the structure of the Tate curve to study the Fourier coefficients of modular forms.

In this note, we adapt the method of Boylan and Ono to provide a simpler, more explicit proof of Ahlgren and Ono's conjecture in the case $A=3^s$. We rely on Bella\"iche and Khare's generalization \cite{bellaiche-khare} of Serre's explicit description \cite{nicolas-serre1, nicolas-serre2} of the Hecke algebra for modulo $2$ reductions of level 1 modular forms. We are then able to show the following.

\begin{theorem}\label{mainthm3}
Let $s$ and $a$ be positive integers. There are infinitely many positive integers $n\equiv a\pmod {3^s}$ such that $p(n)\not\equiv 0\pmod 3$.
\end{theorem}

In Section~\ref{nilpotency}, we discuss Bella\"iche and Khare's work. In Section~\ref{proofmain}, we prove Theorem~\ref{mainthm3}.

\section{Local nilpotency}
\label{nilpotency}

Throughout, let $\Delta=\Delta(z)$ denote the discriminant modular form, and $\tau(n)$ denote the Fourier coefficients of $\Delta$. Let $p$ be a prime number, and let $T_p$ be the $p$th Hecke operator, which by definition acts on a modular form $f(z)=\sum_{n=0}^\infty a(n)q^n$ of weight $k$ by
\[
T_p f(z)=\sum_{n=0}^\infty(a(np)+p^{k-1}a(n/p))q^n
\]
where $a(x)=0$ if $x\notin\mathbb{Z}$. In~\cite{nicolas-serre1}, Nicolas and Serre compute for each $f$ the minimal integer $g=g(f)$ such that for any $g$ primes $p_1,\dotsc,p_g$, we have
\[
T_{p_1}\dotsb T_{p_g} f(z)\equiv0\pmod{2}.
\]
The statement that the algebra generated by the Hecke operators is locally nilpotent modulo $2$ just means that $g(f)$ is finite for every $f$. Additionally, Nicolas and Serre prove in~\cite{nicolas-serre2} that each $T_p$ can be written as a power series in $T_3$ and $T_5$, thereby permitting the enumeration of all sets of primes $p_1,\dotsc,p_{g-1}$ such that $T_{p_1}\dotsb T_{p_{g-1}} f(z)$ is \emph{not} zero modulo $2$. As a result of this description of the Hecke algebra modulo $2$, Boylan and Ono's method in~\cite{boylan-ono} becomes completely explicit.

Since our goal is to recreate Boylan and Ono's work in the modulo $3$ case, we replace Nicolas and Serre's conclusions with the following extension by Bella\"iche and Khare. Write $T_p'=T_p$ if $p\equiv2\pmod3$, and $T_p'=1+T_p$ if $\ell\equiv1\pmod3$. Then the operators $T_p'$ act locally nilpotently on the ring $S(\mathbb{F}_3)$ of level 1 cusp forms with integer coefficients taken modulo 3, in the sense that given such a modular form $f$, there is some minimal integer $g$ such that $T'_{p_1}\cdots T'_{p_g} f=0$ for any sequence of Hecke operators $T_{p_1}',\ldots, T_{p_g}'$. As such, since $\Delta$ and $-\Delta$ are the only cusp forms modulo 3 satisfying $T_p'\vert f=0$ for all $p$, we have some maximal sequence $p_1,\ldots, p_{g-1}$ such that $T_{p_1}'\cdots T_{p_{g-1}}' f=\pm\Delta$. Moreover, we have the following description of the Hecke algebra on modular forms with coefficients in $\mathbb F_3$.

\begin{theorem}[\cite{bellaiche-khare}, Theorem 24]\label{bk}


 The algebra of Hecke operators on $S(\fld_3)$ is isomorphic to the power series ring $\fld_3[[x,y]]$, with an isomorphism given by sending $T_2'=T_2$ to $x$ and $T_7'=1+T_7$ to $y$. 
 Assuming this identification, we have $T_p'\equiv x\pmod{(x,y)^2}$ if and only if $p$ is congruent to $2\pmod3$ but not $8\pmod9$, 
$T_p'\equiv y\pmod{(x,y)^2}$ if and only if $p$ is congruent to $1\pmod3$ and not split in the splitting field of $X^3-3$, and otherwise $T_p'\equiv 0\pmod{(x,y)^2}$.
\end{theorem} 
This theorem in particular implies that for any nonzero $f\in S(\fld_3)$ there are some positive integers $k$ and $\ell$ such that $(T_{2}')^k (T_7')^\ell  f=\pm\Delta$. We may now proceed according to Boylan and Ono's strategy.

\section{Partitions modulo 3} \label{proofmain}
In this section we prove Theorem \ref{mainthm3}. We start by proving a basic lemma, similar to Corollary 1.4 of \cite{radu}.
\begin{lemma}
\label{manyfromone}
Suppose that for every $s$, every arithmetic progression modulo $3^s$ contains at least one $N$ such that $p(N)\not\equiv0\pmod{3}$. Then  for every $s$, every arithmetic progression modulo $3^s$ contains infinitely many such $N$.
\end{lemma}


\begin{proof}
We prove the contrapositive. Suppose there exist some $r$ and $s$, with $0\le r<3^s$, for which there are only finitely many $N\equiv r\pmod{3^s}$ with $p(N)\not\equiv0\pmod3$. Then there is some $a_0\ge0$ such that $p(3^sa+r)\equiv0\pmod3$ for all $a\ge a_0$. Let $t$ be such that $3^t>3^sa_0+r$. Then we have 
\[
p(3^s\cdot 3^ta+3^sa_0+r)\equiv0\pmod3
\]
for all $a\ge0$, from which we conclude that $p(N)\equiv0\pmod3$ for all $N\equiv 3^sa_0+r\pmod{3^{s+t}}$.
\end{proof}
Let the integers $a_s(n)$ be defined by the generating function $\sum_{n=0}^{\infty}a_s(n)q^n=\Delta^{\frac{9^s-1}8}$, and let $r_s(j)$ be defined by $1+\sum_{n=1}^\infty r_s(j)q^{3\cdot9^sj}=\prod_{n=1}^\infty(1-q^{3\cdot9^s})$. Our next lemma is similar to Lemma 2.1 of~\cite{boylan-ono}.

\begin{lemma}
\label{recursion}
We have
\[
a_s(n)\equiv p\left(\frac{n-\frac{9^s-1}8}{3}\right)+\sum_{j=1}^\infty r_s(j)p\left(\frac{n-\frac{9^s-1}8}3-9^sj\right)\pmod3.
\]
\end{lemma}

\begin{proof}
We may compute
\begin{align*}
\Delta^{\frac{9^s-1}8} &= \left(q\prod_{n=1}^\infty(1-q^n)^{24}\right)^{\frac{9^s-1}8} \\
&\equiv q^{\frac{9^s-1}8}\prod_{n=1}^\infty(1-q^{3\cdot 9^sn})\prod_{n=1}^\infty\frac1{1-q^{3n}} \pmod 3 \\
&\equiv \left(1+\sum_{j=1}^\infty r_s(j)q^{3\cdot 9^sj}\right)\left(\sum_{k=0}^\infty p(k)q^{3k+\frac{9^s-1}8}\right) \pmod 3.
\end{align*}
The lemma follows by comparing coefficients.
\end{proof}
Finally, we require the following lemma concerning nonzero coefficients of $\Delta^{\frac{9^s-1}8}$.
\begin{lemma}
\label{nonzero}
There are fixed nonnegative integers $k$ and $\ell$, not both zero, such that the following holds. Let $p_1,\dotsc,p_k$ be distinct primes which satisfy $p_i\equiv 2\pmod3$ and $p_i\not\equiv 8\pmod9$ for all $i$ . Let $q_1,\dotsc,q_\ell$ be distinct primes such that for all $j$ we have $q_j\equiv 1\pmod{9^{s+1}}$ and $q_j$ does not split in the splitting field of $X^3-3$. Then for any $n_0$ satisfying $(n_0, p_1\cdots p_k q_1\cdots q_\ell)=1$  such that $\tau(n_0)\not\equiv 0\pmod3$, there is some $d|q_1\dotsb q_\ell$ for which we have $a_s(n_0p_1\dotsb p_kd)\not\equiv0\pmod3$. 
\end{lemma}
\begin{remark}
Given values of $k$, $\ell$, and $s$, it is always possible to find corresponding primes $p_1,\ldots,p_k$ and $q_1,\ldots, q_\ell$; indeed by the Chebotarev density theorem, the primes that are valid choices for $p_i$ have density $1/3$ within the primes, and the primes that are valid choices for $q_j$ have density $1/(2\cdot 3^{s+1})$.
\end{remark} 
\begin{proof}[Proof of Lemma~\ref{nonzero}.]
Identify the Hecke algebra modulo 3 with the ring $\fld_3[[x,y]]$ as described in Theorem \ref{bk} and let $p_1,\dotsc,p_k,q_1,\dotsc,q_\ell$ be such that $T_{p_i}'\equiv x\pmod{(x,y)^2}$ and $T_{q_j}'\equiv y\pmod{(x,y)^2}$, and that 
\[
T_{p_1}'\dotsb T_{p_k}'T_{q_1}'\dotsb T_{q_\ell}'\Delta(z)^{\frac{9^s-1}8}\equiv\pm\Delta\pmod{3}.
\]

Note that having fixed $k$ and $l$ so that such $p_i$s and $q_j$s exist, any such sequence of primes satisfying the hypotheses of Lemma~\ref{nonzero} satisfy the same equation. By comparing coefficients, we conclude that for any $n_0$ not divisible by the $p_i$s and $q_j$s with a nonzero corresponding coefficient of $\Delta$, we have
\begin{equation}
\sum_{d|q_1\dotsb q_\ell}a_s(n_0p_1\dotsb p_kq_1\dotsb q_\ell)\not\equiv0\pmod 3.
\end{equation}
Lemma ~\ref{nonzero} follows.
\end{proof}

Finally, combining Lemma~\ref{nonzero} with Lemmas~\ref{manyfromone} and~\ref{recursion}, we prove Theorem~\ref{mainthm3}.

\begin{proof}[Deduction of Theorem~\ref{mainthm3} from Lemma~\ref{nonzero}.]
Choose $k$ and $\ell$ such that $(T_2')^k (T_7')^\ell \Delta^{\frac{9^s-1}{8}}=\pm \Delta$ let $p_1,\ldots,p_k$, $q_1,\ldots, q_\ell$ be as in the statement of Lemma \ref{nonzero}. Then, by Lemma \ref{nonzero}, given any $n_0$ with $\tau(n_0)\not\equiv 0\pmod 3$ and $(n_0,p_1\cdots p_k q_1\cdots q_\ell)=1$ we have some $d\vert q_1\cdots q_\ell$ such that $a_s(n_0p_1\cdots p_k d)\not\equiv 0\pmod 3$. By Lemma~\ref{recursion}, for this $d|q_1\dotsb q_\ell$, we may write
\[
a_s(n_0p_1\dotsb p_kd)\equiv p\left(\frac{n_0p_1\dotsb p_kd-\frac{9^s-1}8}{3}\right)+\sum_{j=1}^\infty r_s(j)p\left(\frac{n_0p_1\dotsb p_kd-\frac{9^s-1}8}3-9^sj\right)\pmod3.
\]
Since $a_s(p_1\dotsb p_kd)\not\equiv0\pmod3$, we conclude that $p\left(\frac{p_1\dotsb p_kd-\frac{9^s-1}8}3-9^sj\right)\not\equiv0\pmod3$ for some $j\geq 0$, so for some $n$ satisfying $n\equiv \frac{p_1\dotsb p_kd-\frac{9^s-1}8}3 \pmod{9^s}$ we have $p(n)\not \equiv 0\pmod 3$.  Hence, by Lemma \ref{manyfromone} it suffices to show that as we vary $n_0$ the quantity $\frac{n_0p_1\dotsb p_kd-\frac{9^s-1}8}3$ covers all residue classes modulo $9^s$. Since $d\equiv1\pmod{9^{s+1}}$ regardless of its precise factorization, it may be dropped. To show that $\frac{n_0p_1\dotsb p_k-\frac{9^s-1}8}3$ covers all residue classes as we vary $n_0$, we note first note the standard congruence for the $\tau$ function,
$$
\tau(n)\equiv \sigma_1(n) \pmod 3,
$$
is valid for all $n\not \equiv 0\pmod 3$, and in particular $\tau(p)\equiv 2\pmod 3$ when $p\equiv 1\pmod 3$. In addition, we note that $k$ must be even, since $a_s(n)\not\equiv 0\pmod 3$ implies $n\equiv 1 \pmod 3$, but if $k$ is odd we have $a_s(p_1\cdots p_k D)\not\equiv 0 \pmod 3$ even though $p_1\cdots p_k D \equiv 2 \pmod 3$, a contradiction. So we have that $p_1\cdots p_k\equiv 1 \pmod 3$, so to show that the terms $\frac{n_0p_1\dotsb p_k-\frac{9^s-1}8}3$ cover all residue classes we need only find $n_0$ in all residue classes $n\pmod{ 9^{s+1}}$ satisfying $n\equiv 1\pmod 3$. But to do so we may simply take $n_0$ to be any prime satisfying $n_0\equiv n \pmod {9^{s+1}}$, of which there are infinitely many. Theorem \ref{mainthm3} follows by applying Lemma \ref{manyfromone}.
\end{proof}

\nocite{*}
\bibliographystyle{plain}
\bibliography{biblio}
\end{document}